\documentclass[twoside, 10pt]{article}
\usepackage{amssymb, amsmath, mathrsfs, amsthm}
\usepackage{graphicx}
\usepackage{color}
\usepackage{float}

\newcommand{\bd}{\begin{description}}
\newcommand{\ed}{\end{description}}
\newcommand{\bi}{\begin{itemize}}
\newcommand{\ei}{\end{itemize}}
\newcommand{\be}{\begin{enumerate}}
\newcommand{\ee}{\end{enumerate}}
\newcommand{\beq}{\begin{equation}}
\newcommand{\eeq}{\end{equation}}
\newcommand{\beqs}{\begin{eqnarray*}}
\newcommand{\eeqs}{\end{eqnarray*}}
\newcommand{\flr}[1]{\left\lfloor #1 \right\rfloor}
\newcommand{\ceil}[1]{\left\lceil #1 \right\rceil}

\definecolor{DarkGreen}{rgb}{0.2, 0.6, 0.3}

\newtheorem{theorem}{Theorem}

\newtheorem{lemma}{Lemma}

\newtheorem{corollary}[theorem]{Corollary}
\newtheorem{case}{Case}

\newtheorem{claim}{Claim}
\newtheorem{fact}{Fact}

\begin{document}
\title{\textbf{Ramsey and Gallai-Ramsey numbers for two classes of unicyclic graphs}\footnote{Supported by the National Science Foundation of China (Nos. 11601254, 11551001, 11161037, 61763041, 11661068, and 11461054) and the Science Found of Qinghai Province (Nos.  2016-ZJ-948Q, and 2014-ZJ-907) and the  Qinghai Key Laboratory of Internet of Things Project (2017-ZJ-Y21).}}

\author{
Zhao Wang\footnote{College of Science, China Jiliang University, Hangzhou 310018, China. {\tt wangzhao@mail.bnu.edu.cn}},
Yaping Mao\footnote{School of Mathematics and Statistis, Qinghai Normal University, Xining, Qinghai 810008, China. {\tt maoyaping@ymail.com}},
\\
Colton Magnant\footnote{Department of Mathematics, Clayton State University, Morrow, GA, 30260, USA. {\tt dr.colton.magnant@gmail.com}},
Jinyu Zou\footnote{School of Computer, Qinghai Normal University, Xining, Qinghai 810008, China. {\tt zjydjy2015@126.com}}.
}


\maketitle

\begin{abstract}
Given a graph $G$ and a positive integer $k$, define the \emph{Gallai-Ramsey number} to be the minimum number of vertices $n$ such that any $k$-edge coloring of $K_n$ contains either a rainbow (all different colored) triangle or a monochromatic copy of $G$. In this paper, we consider two classes of unicyclic graphs, the star with an extra edge and the path with a triangle at one end. We provide the $2$-color Ramsey numbers for these two classes of graphs and use these to obtain general upper and lower bounds on the Gallai-Ramsey numbers.
\end{abstract}

\section{Introduction}

In this work, we consider only edge-colorings of graphs. A coloring of a graph is called
\emph{rainbow} if no two edges have the same color.

Colorings of complete graphs that contain no rainbow triangle have very interesting and somewhat surprising structure. In 1967, Gallai \cite{MR0221974} first examined this structure
under the guise of transitive orientations. The result was reproven in \cite{MR2063371} in the terminology of graphs and can also be traced
to \cite{MR1464337}. For the following statement, a trivial partition is a partition into only one part.

\begin{theorem}[\cite{MR1464337,MR0221974,MR2063371}]\label{Thm:G-Part}
In any coloring of a complete graph containing no rainbow
triangle, there exists a nontrivial partition of the vertices (that is, with at least two parts) such
that there are at most two colors on the edges between the parts and only one color on
the edges between each pair of parts.
\end{theorem}

For ease of notation, we refer to a colored complete graph with no rainbow triangle as a \emph{Gallai-coloring} and the partition provided by Theorem~\ref{Thm:G-Part} as a \emph{Gallai-partition}. The induced subgraph of a Gallai colored complete graph constructed by selecting a single vertex from each part of a Gallai partition is called the \emph{reduced graph} of that partition. By Theorem~\ref{Thm:G-Part}, the reduced graph is a $2$-colored complete graph.

Given two graphs $G$ and $H$, let $R(G, H)$ denote the $2$-color Ramsey number for finding a monochromatic $G$ or $H$, that is, the minimum number of vertices $n$ needed so that every red-blue coloring of $K_{n}$ contains either a red copy of $G$ or a blue copy of $H$. Although the reduced graph of a Gallai partition uses only two colors, the original Gallai-colored complete graph could certainly use more colors. With this in mind, we consider the following generalization of the Ramsey numbers. Given two graphs $G$ and $H$, the \emph{general $k$-colored Gallai-Ramsey number} $gr_k(G:H)$ is defined to be the minimum integer $m$ such that every $k$-coloring of the complete graph on $m$ vertices contains either a rainbow copy of $G$ or a monochromatic copy of $H$. With the additional restriction of forbidding the rainbow copy of $G$, it is clear that $gr_k(G:H)\leq R_k(H)$ for any graph $G$.

In this work, we consider the Gallai-Ramsey numbers for finding either a rainbow triangle or monochromatic graph coming from two classes of unicyclic graphs: a star with an extra edge that forms a triangle, and a path with an extra edge from an end vertex to an internal vertex formaing a triangle. In order to produce sharp results for the Gallai-Ramsey numbers of these graphs, we first prove the $2$-color Ramsey numbers for these graphs.

These graphs are particularly interesting because although they are not bipartite, they are very close to being a tree (and therefore bipartite). The dichotomy between bipartite and non-bipartite graphs is critical in the study of Gallai-Ramsey numbers in light of the following result.

\begin{theorem}[\cite{GSSS10}]
Let $H$ be a fixed graph with no isolated vertices.  If $H$ is not bipartite, then $gr_{k}(K_{3} : H)$ is exponential in $k$.   If $H$ is bipartite, then $gr_{k}(K_{3} : H)$ is linear in $k$.
\end{theorem}

We refer the interested reader to \cite{MR1670625} for a dynamic survey of small Ramsey numbers and \cite{FMO14} for a dynamic survey of rainbow generalizations of Ramsey theory, including topics like Gallai-Ramsey numbers.

Section~\ref{Section:Prelim} contains two known results that will be used as part of our proofs. Section~\ref{Section:StarPlus} discusses the case where $H$ is a star with the addition of an extra edge to form a triangle. Section~\ref{Section:PathPlus} discusses the case where $H$ is a path with the addition of an extra edge betwen the first vertex and the third vertex, again forming a triangle.

\section{Preliminaries}\label{Section:Prelim}

In this section, we recall two results about cycles and a helpful lemma which will be used later in our proofs. First the known Ramsey number for cycles.

\begin{theorem}[\cite{MR0345866, MR1846919, MR0332567}]\label{Thm:RamseyCycles}
Given integers $m, n \geq 3$,
$$
R(C_{m}, C_{n}) = \begin{cases}
2n - 1 \\
~ ~ ~ ~ \text{ for $3 \leq m \leq n$, $m$ odd, $(m, n) \neq (3, 3)$,}\\
n - 1 + m/2 \\
~ ~ ~ ~ \text{ for $4 \leq m \leq n$, $m$ and $n$ even, $(m, n) \neq (4, 4)$,}\\
\max\{ n - 1 + m/2, 2m - 1\}\\
~ ~ ~ ~ \text{ for $4 \leq m < n$, $m$ even and $n$ odd.}
\end{cases}
$$
\end{theorem}


Next the general $k$-color Gallai-Ramsey number for even cycles is not yet known but the following bound have been shown.

\begin{theorem}[\cite{FM11}, \cite{MR3121112}] \label{Thm:GR-EvenCycles}
Given integers $n \geq 2$ and $k \geq 1$,
$$
(n - 1)k + n + 1 \leq gr_{k} (K_{3} : C_{2n}) \leq (n - 1)k + 3n.
$$
\end{theorem}

Finally we present a lemma from \cite{OddCycles}.

\begin{lemma}[\cite{OddCycles}]\label{Lemma:SmallOutside}
Let $k \geq 3$, $H$ be a graph with $|H| = m$, and let $G$ be a Gallai coloring of the complete graph $K_{n}$ containing no monochromatic copy of $H$. If $G = A \cup B_{1} \cup B_{2} \cup \dots \cup B_{k - 1}$ where $A$ uses at most $k$ colors (say from $[k]$), $|B_{i}| \leq m - 1$ for all $i$, and all edges between $A$ and $B_{i}$ have color $i$, then $n \leq gr_{k}(K_{3} : H) - 1$.
\end{lemma}

Note that this lemma uses the assumed structure to provide a bound on $|G|$ even if $G$ itself uses more than $k$ colors.


\section{Star with an extra edge}\label{Section:StarPlus}

Let $S_{t}$ denote the star with $t$ total vertices (and $t - 1$ edges). Then for $t \geq 3$, let $S_{t}^{+}$ denote graph consisting of the star $S_{t}$ with the addition of an edge between two of the pendant vertices, forming a triangle. Note that $S_{3}^{+} = K_{3}$.

Before beginning the discussion of the Gallai-Ramsey number for $S_{t}^{+}$, we must first find the $2$-color Ramsey number.

\begin{theorem}\label{Thm:RamseySt+}
For $t \geq 3$,
$$
R(S_{t}^{+}, S_{t}^{+}) = 2t - 1.
$$
\end{theorem}

\begin{proof}
The lower bound follows from the graph constructed by taking two copies of $K_{t-1}$ each colored entirely with red and adding all blue edges in between the two copies. Each red component is too small to contain a monochromatic copy of $S_{t}^{+}$ and the blue subgraph is bipartite so cannot contain a copy of $S_{t}^{+}$. The order of this constructed graph is $2t-2$, meaning that the Ramsey number is at least $2t-1$.

For the upper bound, consider an arbitrary vertex $v$ in any $2$-coloring of $K_{2t - 1}$. If $v$ has at least $t - 1$ incident red edges, then the graph induced on the set of vertices at the opposite end of these red edges must contain no red edges to avoid a red $S_{t}^{+}$. This induces a blue complete graph so there must be at most $t - 1$ such vertices. The same holds for incident blue edges at $v$, meaning that $v$ must have precisely $t - 1$ incident red edges and $t - 1$ incident blue edges. Let $R$ and $B$ be the corresponding sets of vertices, as above, inducing blue and red complete graphs respectively and let $r \in V(R)$ and $b \in V(B)$. The edge $rb$ must be either red or blue, but either one creates the desired monochromatic copy of $S_{t}^{+}$ centered at one of $r$ or $b$, completing the proof.
\end{proof}

Next we prove a lemma which provides the lower bound on the Gallai-Ramsey number.

\begin{lemma}\label{Lem:GRSt+}
For $k \geq 1$,
$$
gr_{k}(K_{3}:S_{t}^{+})\geq\begin{cases}
2(t-1)\cdot5^{\frac{k-2}{2}}+1 & \text{ if $k$ is even,}\\
(t-1)\cdot5^{\frac{k-1}{2}}+1 & \text{ if $k$ is odd.}
\end{cases}
$$
\end{lemma}

\begin{proof}
We prove this result by inductively constructing a $k$-coloring of $K_{n}$ where
$$
n=\begin{cases}
2(t-1)\cdot5^{\frac{k-2}{2}} & \text{ if $k$ is even,}\\
(t-1)\cdot5^{\frac{k-1}{2}} & \text{ if $k$ is odd,}
\end{cases}
$$
which contains no rainbow triangle and no monochromatic copy of $S_{t}^{+}$.

If $k$ is odd, let $G_{1}$ be a complete graph on $t - 1$ vertices colored entirely with color $1$. Note that with only $t - 1$ vertices, this contains no monochromatic copy of $S_{t}^{+}$. 
Suppose we have constructed a colored complete graph $G_{2i-1}$ where $i$ is a positive integer and $2i-1<k$, using the $2i-1$ colors 1,2, \ldots, $2i-1$ and having order $n_{2i-1}=(t-1)\cdot5^{i-1}$. Construct $G_{2i+1}$ by making five copies of $G_{2i-1}$ and inserting edges of color $2i$ and $2i+1$ between the copies to form a blow-up of the unique 2-colored $K_{5}$ which contains no monochromatic triangle. This coloring clearly contains no rainbow triangle and, since there is no monochromatic triangle in either of the two new colors, there can be no monochromatic copy of $S_{t}^{+}$ in $G_{2i+1}$. Ultimately, $G_{k}$ is a $k$-colored complete graph containing no rainbow triangle and no monochromatic copy of $S_{t}^{+}$ with $|G_{k}| = (t - 1) \cdot 5^{(k - 1)/2}$.

If $k$ is even, let $G_{2}$ be a $2$-colored complete graph on $2t-2$ vertices containing no monochromatic copy of $S_{t}^{+}$ using colors $1$ and $2$. That is, $G_{2}$ is the sharpness example from Theorem~\ref{Thm:RamseySt+}. Suppose we have constructed a coloring of $G_{2i}$ where $i$ is a positive integer and $2i < k$, using $2i$ colors 1, 2, \ldots, $2i$ and having order $n_{2i} = (2t-2) \cdot 5^{i - 1}$ such that $G_{2i}$ contains no rainbow triangle and no monochromatic copy of $S_{t}^{+}$. Construct $G_{2i + 2}$ by making five copies of $G_{2i}$ and inserting edges of colors $2i + 1$ and $2i + 2$ between the copies to form a blow-up of the unique $2$-colored $K_{5}$ which contains no monochromatic triangle. This coloring clearly contains no rainbow triangle and, since there is no monochromatic triangle in either of the two new colors, there can be no monochromatic copy of $S_{t}^{+}$ in $G_{2i + 2}$. Ultimately, $G_{k}$ is a $k$-colored complete graph containing no rainbow triangle and no monochromatic copy of $S_{t}^{+}$ with $|G_{k}| = 2(t - 1) \cdot 5^{(k - 2)/2}$.
\end{proof}

At last, the main result of this section, the precise Gallai-Ramsey number for $S_{t}^{+}$.

\begin{theorem}\label{Thm:GR-StarPlus}
For $k \geq 1$ and $t\geq3$,
$$
gr_{k}(K_{3}:S_{t}^{+})=\begin{cases}
2(t-1)\cdot5^{\frac{k-2}{2}}+1 & \text{ if $k$ is even,}\\
(t-1)\cdot5^{\frac{k-1}{2}}+1 & \text{ if $k$ is odd.}
\end{cases}
$$
\end{theorem}

\begin{proof}
The lower bound follows from Lemma~\ref{Lem:GRSt+}. We prove the upper bound by induction on $k$. The case $k=1$ is trivial and the case $k=2$ is precisely Theorem~\ref{Thm:RamseySt+}, so suppose $k\geq3$ and let $G$ be a Gallai coloring of $K_{n}$ where
$$
n=\begin{cases}
2(t-1)\cdot5^{\frac{k-2}{2}}+1 & \text{ if $k$ is even,}\\
(t-1)\cdot5^{\frac{k-1}{2}}+1 & \text{ if $k$ is odd.}
\end{cases}
$$

Since $G$ is a Gallai coloring, by Theorem~\ref{Thm:G-Part}, there is a Gallai partition of $G$. Suppose red and blue are the two colors appearing in the Gallai partition. Let $m$ be the number of parts in this partition and choose such a partition where $m$ is minimized. By Theorem~\ref{Thm:RamseySt+} applied on the reduced graph, we see that $m\leq2t-2$. Let $H_{1}, H_{2}, \dots, H_{m}$ be the parts of this partition and suppose $|H_{i}| \geq |H_{i + 1}|$ for all $i$ with $1 \leq i \leq m - 1$. Let $r$ be the number of parts of the Gallai partition with order at least $t-1$, so $|H_{r}|\geq t-1$ and $|H_{r+1}| \leq t-2$.

First suppose $2\leq m\leq3$. If $m\leq3$, then by the minimality of $m$, we may assume $m=2$, say with corresponding parts $H_{1}$ and $H_{2}$. Without loss of generality, suppose all edges between $H_{1}$ and $H_{2}$ are blue. Since $k\geq3$, we have $|G|=|H_{1}|+|H_{2}|\geq 5t-4$, so there is at least one part of order at least $t-1$, meaning that $|H_{1}|\geq t-1$. If $|H_{i}|\geq t-1$ for $i=1, 2$, then to avoid a blue $S_{t}^{+}$, there can be no blue edges within $H_{1}$ and $H_{2}$. Since a color is missing within each $H_{i}$, apply induction on $k$ within $H_{i}$. This means that
$$
|G|=|H_{1}|+|H_{2}|\leq 2[gr_{k-1}(K_{3}:S_{t}^{+})-1]<n,
$$
a contradiction. Otherwise if $|H_{2}| < t - 1$, then still there are no blue edges within $H_{1}$ so by induction on $k$ within $H_{1}$, we have
$$
|G|=|H_{1}|+|H_{2}|\leq [gr_{k-1}(K_{3}:S_{t}^{+})-1] + (t - 2) < n,
$$
a contradiction. We may therefore assume $m \geq 4$.

If $r\geq 4$ and $m\geq 6$, then any choice of $6$ parts containing the $4$ parts
$\mathcal{H}=\{H_{1}, H_{2}, H_{3}, H_{4}\}$ will contain a monochromatic triangle in the reduced graph. Such a triangle must contains at least one part from $\mathcal{H}$, meaning that the corresponding subgraph of $G$ must contain a monochromatic copy of $S_{t}^{+}$, a contradiction. Thus, we may assume either $4\leq m\leq5$ or $r\leq 3$. We consider cases based on the value of $r$. First a couple of claims that will be helpful in the later analysis.

\begin{claim}\label{Clm:AllSmall}
If there are several parts of a Gallai partition, each of order at most $t - 2$, such that all edges in between pairs of these parts are red, then there are at most a total of $2t - 4$ vertices in these parts.
\end{claim}

\begin{proof}
Let $H'_{1}, H'_{2}, \dots, H'_{m'}$ be these parts. If $m' \leq 2$, then $|H'_{1}\cup H'_{2}| \leq 2t-4$ by assumption. If $m' \geq 3$, to avoid creating a red $S_{t}^{+}$ using a vertex of $H'_{1}$ as the center of the star, we have $|H'_{2} \cup H'_{3} \cup \dots \cup H'_{m'}| \leq t-2$, so $|H'_{1}\cup H'_{2}\cup \dots \cup H'_{m'}| \leq 2t-4$, completing the proof.
\end{proof}

\begin{claim}\label{Clm:DifferentColors}
If $H_{1}$ and $H_{2}$ are two parts of a Gallai partition each with order at least $t - 1$, say with blue edges in between $H_{1}$ and $H_{2}$, then there is at most one part with blue edges to $H_{1}$ and with red edges to $H_{2}$, and similarly, there is at most one part with red edges to $H_{1}$ and with blue edges to $H_{2}$. If $c$ is the color of the edges between $H_{1}$ and $H_{2}$, then there are no parts of the Gallai partition with edges of color $c$ to both $H_{1}$ and $H_{2}$.
\end{claim}

\begin{proof}
For a contradiction, suppose there are two parts $H_{3}$ and $H_{4}$ with red edges to $H_{1}$ and blue edges to $H_{2}$. If the edges from $H_{3}$ to $H_{4}$ are red, then $H_{1} \cup H_{3} \cup H_{4}$ contains a red copy of $S_{t}^{+}$. Otherwise, if the edges from $H_{3}$ to $H_{4}$ are blue, then $H_{2} \cup H_{3} \cup H_{4}$ contains a blue copy of $S_{t}^{+}$, either case providing a contradiction. The proof is symmetric for two parts with red edges to $H_{2}$ and blue edges to $H_{1}$.

For the second conclusion, if there was a part $H_{3}$ with blue edges to both $H_{1}$ and $H_{2}$, then $H_{1} \cup H_{2} \cup H_{3}$ contains a blue copy of $S_{t}^{+}$, for a contradiction.
\end{proof}

\setcounter{case}{0}
\begin{case}
$r = 0$.
\end{case}

Consider the colors of the edges from $H_{1}$ to $H_{2} \cup H_{3} \cup \dots \cup H_{m}$. 
Let $A$ be the union of the parts with blue edges to $H_{1}$ and let $B$ be the union of the parts with red edges to $H_{1}$. If $|A|$ (or similarly $|B|$) is at least $t - 1$, then there can be no blue edges within $A$ (respectively red edges within $B$). Then all edges between the parts in $A$ (respectively $B$) must be red (respectively blue) so by Claim~\ref{Clm:AllSmall}, we have $|A| \leq 2t - 4$ and $|B| \leq 2t - 4$. This means
$$
|G|=|H_{1}|+|A|+|B|\leq 5(t-2)<n,
$$
a contradiction.

\begin{case}
$r=1$.
\end{case}

Again let $A$ be the union of the parts with blue edges to $H_{1}$ and let $B$ be the union of the parts with red edges to $H_{1}$. As in the previous case, we have $|A| \leq 2t - 4$ and $|B| \leq 2t - 4$. Since $m$ is minimal, neither $A$ nor $B$ can be empty. There can therefore be no red or blue edges within $H_{1}$, we have
\beqs
|G| & = & |H_{1}|+|A|+|B|\\
~ & \leq & [gr_{k - 2}(K_{3} : S_{t}^{+}) - 1] + 2(2t-4)\\
~ & < & n,
\eeqs
a contradiction. 

\begin{case}
$r=2$.
\end{case}

Suppose blue is the color of the edges between $H_{1}$ and $H_{2}$, so neither $H_{1}$ nor $H_{2}$ can contain blue edges and by Claim~\ref{Clm:DifferentColors}, there is no part with blue edges to both $H_{1}$ and $H_{2}$. By Claim~\ref{Clm:DifferentColors}, there is at most one part $H_{3}$ with blue edges to $H_{1}$ and with red edges to $H_{2}$, and there is at most one part $H_{4}$ with red edges to $H_{1}$ and with blue edges to $H_{2}$. Let $A$ be the union of the remaining parts, with all red edges to $H_{1}\cup H_{2}$. By Claim~\ref{Clm:AllSmall}, we have $|A|\leq 2t-4$. By the minimality of $m$, all parts have incident edges from other parts in both red and blue so both $H_{1}$ and $H_{2}$ contain no red edges or blue edges. This means $|H_{i}|\leq gr_{k-2}(K_{3}:S_{t}^{+})-1$, so
\begin{eqnarray*}
|G| & = & |H_{1}|+|H_{2}|+|H_{3}|+|H_{4}|+|A|\\
~ & \leq & 2[gr_{k-2}(K_{3}:S_{t}^{+})-1]+4(t-2)\\
 & < & n,
\end{eqnarray*}
a contradiction. 

\begin{case}
$r=3$.
\end{case}

To avoid a monochromatic copy of $S_{t}^{+}$, the triangle in the reduced graph corresponding to the parts $\{H_{1}, H_{2}, H_{3}\}$ must not be monochromatic. Without loss of generality, suppose the edges from $H_{2}$ to $H_{3}$ are red and all edges from $H_{1}$ to $H_{2} \cup H_{3}$ are blue. Then $H_{2}$ and $H_{3}$ contain neither red nor blue edges, and $H_{1}$ contains no blue edges.

First we claim that there is no part having blue edges to $H_{1}$. Otherwise suppose there is such a part, say $H'$ with blue edges to $H_{1}$. To avoid a blue copy of $S_{t}^{+}$, all edges from $H'$ to $H_{2}\cup H_{3}$ must be red, but then $H' \cup H_{2} \cup H_{3}$ contains a red copy of $S_{t}^{+}$, a contradiction. Thus, all edges from $H_{1}$ to $H_{4}\cup \ldots \cup H_{m}$ must be red. Since $m \geq 4$, this means $H_{1}$ also contains no red edges so $|H_{i}|\leq gr_{k-2}(K_{3}:S_{t}^{+})-1$ for $1\leq i \leq3$.

By Claim~\ref{Clm:DifferentColors},  there is at most one part with blue edges to $H_{2}$ and red edges to $H_{3}$, say $H_{4}$, and there is at most one part with red edges to $H_{2}$ and blue edges to $H_{3}$, say $H_{5}$. Also by Claim~\ref{Clm:DifferentColors}, there is at most one part with blue edges to $H_{2}\cup H_{3}$, say $H_{6}$. Note that $H_{4} \cup H_{5} \cup H_{6} \neq \emptyset$ by the minimality of $m$. This covering all the possibilities, we get 
\begin{eqnarray*}
|G| & = & \sum_{i=1}^{6}|H_{i}|\\
~ & \leq & 3[gr_{k-2}(K_{3}:S_{t}^{+})-1]+3(t-2)\\
&<&n,
\end{eqnarray*}
a contradiction. 

\begin{case}
$r\geq 4$.
\end{case}

As observed previously, this implies that $4\leq m \leq5$. Within the subgraph of the reduced graph induced on the $r$ parts of order at least $t - 1$, there can be no monochromatic triangle. If $r=5$, there is only one coloring of $K_{5}$ with no monochromatic triangle and if $r=4$, there are two colorings of $K_{4}$ with no monochromatic triangle. In each of these colorings, every vertex has at least one incident edge in both colors, meaning that all of the $r$ corresponding parts of order at least $t-1$ must have no red or blue edges. Then
\begin{eqnarray*}
|G| & = & \sum_{i=1}^{m}|H_{i}|\\
~ & \leq & 5[gr_{k-2}(K_{3}:S_{t}^{+})-1]\\
&<&n,
\end{eqnarray*}
a contradiction, completing the proof of this case and the proof of Theorem~\ref{Thm:GR-StarPlus}.
\end{proof}

\section{Path with a triangle end}\label{Section:PathPlus}

Let $P_{t}$ denote the path of order $t$. Then for $t\geq3$, let $P_{t}^{+}$ denote the graph consisting of the path $P_{t}$ with the addition of an edge between one end and the vertex at distance 2 along the path from that end, forming a triangle. Note that $P_{3}^{+}=K_{3}$ and $P_{4}^{+}=S_{4}^{+}$.

Before beginning the study of the Gallai-Ramsey number, we first establish the $2$-color Ramsey number for $P_{t}^{+}$.

\begin{theorem}\label{Thm:RamseyPt+}
For $t \geq 4$,
$$
R(P_{t}^{+}, P_{t}^{+}) = 2t - 1.
$$
\end{theorem}

Note that if $t = 3$, then $P_{3}^{+} = K_{3}$ so $R(P_{3}^{+}, P_{3}^{+}) = 6$.

\begin{proof}
The lower bound follows from the graph constructed by taking two copies of $K_{t - 1}$ each colored entirely with red and adding all edges in blue in between. Each red component is too small to contain a monochromatic copy of $P_{t}^{+}$ and the blue subgraph is bipartite so cannot contain a copy of $P_{t}^{+}$. The order of this constructed graph is $2t-2$, meaning that the Ramsey number is at least $2t-1$.

If $t = 4$, then $P_{4}^{+} = S_{4}^{+}$ so $R(P_{4}^{+}, P_{4}^{+}) = 7$ by Theorem~\ref{Thm:GR-StarPlus}, so suppose $t \geq 5$.
For the upper bound in general, let $G$ be a $2$-coloring of $K_{2t - 1}$, say using red and blue.

First suppose $t$ is odd, so by Theorem~\ref{Thm:RamseyCycles}, there is a monochromatic copy of $C_{t}$ in $G$, say with $C$ being a red copy of $C_{t}$. If we let $C = v_{1}v_{2}\cdots v_{t}v_{1}$, then to avoid creating a red copy of $P_{t}^{+}$, all edges of the form $v_{i}v_{i + 2}$ must be blue, where indices are taken modulo $t$. Since $t$ is odd, these edges form a blue copy of $C_{t}$. Let $v$ be an arbitrary vertex in $G \setminus C$. Without loss of generality, suppose $vv_{2}$ is red. Then to avoid creating a red copy of $P_{t}^{+}$, the edges $vv_{1}$ and $vv_{3}$ must be blue. Then $vv_{1}v_{3}$ forms a blue triangle and, along with the rest of the blue cycle, this structure contains the desired blue copy of $P_{t}^{+}$.

Next suppose $t$ is even and additionally we first suppose $t \geq 10$. 
Then by Theorem~\ref{Thm:RamseyCycles}, there is a monochromatic copy of $C_{t + 2}$ in $G$, say with $C$ being a red copy of $C_{t + 2}$. We again let $C = v_{1}v_{2}\cdots v_{t + 2}v_{1}$ and note that, as in the previous case, every edge of the form $v_{i}v_{i + 2}$ must be blue where indices are taken modulo $t + 2$. Since $t$ is even, these blue edges induce two copies of $C_{(t + 2)/2}$.

If no vertex in $G \setminus V(C)$ has red edges to $C$, then it is easy to construct a blue copy of $P_{t}^{+}$ so let $v \in G \setminus V(C)$ with $v$ having a red edge to $C$, say to $v_{2}$. Then both edges $vv_{1}$ and $vv_{3}$ must be blue to avoid a red copy of $P_{t}^{+}$. Let $C_{even}$ (and $C_{odd}$) be the two blue cycles on the even (respectively odd) indexed vertices. Since these edges form a blue triangle with the edge $v_{1}v_{3}$, this restricts the blue edges that can go between the two blue cycles. In fact, all edges from $\{v_{t + 1}, v_{5}\}$ to $C_{even}$ must be red. To avoid a red $P_{t}^{+}$, the edges $vv_{5}$ and $vv_{t + 1}$ must also be blue so, as above, $v_{1}$ and $v_{3}$ must also have all red edges to $C_{even}$. In order to avoid a red copy of $P_{t}^{+}$, these red edges imply that all edges between pairs of even indexed vertices must be blue, inducing a blue complete graph of order $\frac{t + 2}{2}$. To avoid a blue copy of $P_{t}^{+}$, all edges between $C_{even}$ and $C_{odd}$ must be red, in turn meaning that all edges between pairs of odd indexed vertices must be blue. Let $A$ and $B$ be the two blue cliques. Each vertex in $G \setminus V(C)$ can have red edges to only one of $A$ or $B$, say $A$. This means that each such vertex must have all blue edges to the opposite clique $B$, which in turn means that it must have all red edges to $A$. Putting all of this together, we see that the red subgraph is a complete bipartite graph. With $|G| = 2t-1$, one part must have order at least $t$ and induce a blue complete graph, containing the desired copy of $P_{t}^{+}$.

If $t = 6$, then $|G| = 11$. By Theorem~\ref{Thm:RamseyCycles}, there is a monochromatic copy of $C_{6}$, say $C$ in red with vertices $v_{1}, v_{2}, \dots, v_{6}$ in this order. As above, every edge of the form $v_{i}v_{i + 2}$ must be blue where indices are taken modulo $t$. These blue edges induce two blue triangles. To avoid creating a blue copy of $P_{6}^{+}$, all edges between these two triangles must be red. To avoid creating a red copy of $P_{6}^{+}$, no vertex in $G \setminus C$ can have at least one red edge to both blue triangles, meaning that every vertex has all blue edges to at least one of the two triangles. Since there are $5$ vertices in $G \setminus C$, at least three of them must have all blue edges to a single blue triangle, say vertices $x, y, z$ have all blue edges to the blue triangle $v_{2}v_{4}v_{6}$. Then the graph induced on $\{x, y, z, v_{2}, v_{4}, v_{6}\}$ contains a blue copy of $P_{6}^{+}$, for a contradiction.

If $t = 8$, then $|G| = 15$. By Theorem~\ref{Thm:RamseyCycles}, there is a monochromatic copy of $C_{8}$, say $C$ in red with vertices $v_{1}, v_{2}, \dots, v_{8}$ in this order. As above, every edge of the form $v_{i}v_{i + 2}$ must be blue where indices are taken modulo $t$. These blue edges induce two blue copies of $C_{4}$ with vertices $v_{1}v_{3}v_{5}v_{7}$ and $v_{2}v_{4}v_{6}v_{8}$ respectively.

First, suppose there is a red edge within each of these copies of $C_{4}$, say without loss of generality, that $v_{1}v_{5}$ and $v_{2}v_{6}$ are red. Then to avoid creating a red copy of $P_{8}^{+}$, the edges $v_{1}v_{4}, v_{2}v_{7}, v_{3}v_{6}$, and $v_{5}v_{8}$ must be blue. These edges together with the two copies of $C_{4}$ form a blue cube $P_{2} \times P_{2} \times P_{2}$. If any face of this cube contains a blue edge, say for example the edge $v_{1}v_{6}$, then there would be a blue copy of $P_{8}^{+}$ with path vertices $v_{1}v_{3}v_{6}v_{4}v_{2}v_{8}v_{7}v_{7}$ and extra edge $v_{1}v_{6}$. Thus, every face of this cube contains only red edges. Let $v$ be an arbitrary vertex in $G \setminus C$. If $v$ has a blue edge to any vertex of $C$, without loss of generality say $v_{1}$, then to avoid creating a blue copy of $P_{8}^{+}$, all edges from $v$ to $\{v_{3}, v_{4}, v_{7}\}$ must be red. This makes a red copy of $P_{8}^{+}$ with path vertices $vv_{3}v_{7}v_{8}v_{1}v_{2}v_{6}v_{5}$ and extra edge $v_{3}v_{7}$. Thus, every vertex in $G \setminus C$ must have only red edges to $C$ but this again creates the same red $P_{8}^{+}$, a contradiction.

Finally, we may assume there is no red edge within one of the copies of $C_{4}$, say without loss of generality, that $A = \{v_{1}, v_{3}, v_{5}, v_{7}\}$ induces a blue copy of $K_{4}$. Any blue edge from $A$ to $B = \{v_{2}, v_{4}, v_{6}, v_{8}\}$ would create a blue copy of $P_{8}^{+}$ so all such edges must be red. This, in turn, means that $B$ also induces a blue copy of $K_{4}$. In order to avoid a red copy of $P_{8}^{+}$, no vertex of $G \setminus C$ can have red edges to both $A$ and $B$, so each vertex in $G \setminus C$ must have all blue edges to at least one of $A$ or $B$. Since there are $7$ vertices in $G \setminus C$, there must be at least $4$ of them with all blue edges to the same set, say $A$. The blue graph induced on these vertices along with $A$ easily contains a blue copy of $P_{8}^{+}$, a contradiction completing the proof.
\end{proof}

In fact, the same proof yields a slightly more general result.

\begin{corollary}\label{Cor:GenRamseyPt+}
For $t \geq s \geq 4$,
$$
R(P_{s}^{+}, P_{t}^{+}) = 2t - 1.
$$
\end{corollary}

We now begin the discussion of the Gallai-Ramsey number for $P_{t}^{+}$ by stating the lower bound. Indeed, the same construction as used in the proof of Lemma~\ref{Lem:GRSt+} also contains no monochromatic copy of $P_{t}^{+}$ so this result is an immediate corollary.

\begin{lemma}\label{Lem:PathPlusLowBd}
For $t \geq 4$ and $k \geq 1$,
$$
gr_{k}(K_{3}:P_{t}^{+})\geq\begin{cases}
2(t-1)\cdot5^{\frac{k-2}{2}}+1 & \text{ if $k$ is even,}\\
(t-1)\cdot5^{\frac{k-1}{2}}+1 & \text{ if $k$ is odd.}
\end{cases}
$$
\end{lemma}

\begin{theorem}
For $t \geq 4$ and $k \geq 1$,
$$
gr_{k}(K_{3}:P_{t}^{+})=\begin{cases}
2(t-1)\cdot5^{\frac{k-2}{2}}+1 & \text{ if $k$ is even,}\\
(t-1)\cdot5^{\frac{k-1}{2}}+1 & \text{ if $k$ is odd.}
\end{cases}
$$
\end{theorem}

\begin{proof}
The case $k = 1$ is trivial. From Theorem~\ref{Thm:RamseyPt+}, we have $R(P_{t}^{+}, P_{t}^{+}) = 2t - 1$, and hence the result is true for $k=2$. We therefore suppose $k\geq 3$ and let $G$ be a coloring of $K_n$ where
$$
n = n_{k} =\begin{cases}
2(t-1)\cdot5^{\frac{k-2}{2}}+1 & \text{ if $k$ is even,}\\
(t-1)\cdot5^{\frac{k-1}{2}}+1 & \text{ if $k$ is odd.}
\end{cases}
$$

Let $T$ be a maximal set of vertices $T = T_{1} \cup T_{2} \cup \dots \cup T_{k}$ where each subset $T_{i}$ has all incident edges to $G \setminus T$ in color $i$ and $|G \setminus T| \geq \ceil{\frac{t}{2}}$ constructed iteratively by adding at most $2\flr{\frac{t}{2}}$ vertices to $T$ at a time, with at most $\flr{\frac{t}{2}}$ vertices being added to each of two sets $T_{i}$ at a time. We first claim that each $|T_{i}|$ is small.

\begin{claim}\label{Clm:TisSmall}
For each $i$ with $1 \leq i \leq k$, we have $|T_{i}| \leq \flr{\frac{t}{2}} - 1$ and furthermore, $T_{i} = \emptyset$ for some $i$.
\end{claim}

Note that this implies that $|T| \leq (k - 1)\left(\flr{\frac{t}{2}} - 1\right)$. The proof of Claim~\ref{Clm:TisSmall} is similar to the corresponding proof in \cite{OddCycles}.

\begin{proof}
By the iterative definition of $T$, we may assume that this is the first step in the iterative construction where the set $T$ violates either of these assumptions. That is, assume that $|T_{i}| \leq 2\flr{\frac{t}{2}} - 1$ for all $i$ and either
\bi
\item at most two sets $T_{i}$ and $T_{j}$ have $|T_{i}|, |T_{j}| > \flr{\frac{t}{2}} - 1$, or
\item no set is empty and at most one set $T_{i}$ has $|T_{i}| > \flr{\frac{t}{2}} - 1$.
\ei
In either case, we certainly have $|T| \leq (k + 1)\flr{\frac{t}{2}}$.

We first show that $T_{i} = \emptyset$ for some $i$ so suppose the latter item above. If $k \geq 4$, then
$$
|G\setminus T| \geq n - (k + 1) \flr{\frac{t}{2}} \geq \left( \frac{t - 1}{2} - 1 \right) k + 3\left( \frac{t - 1}{2} \right).
$$
By Theorem~\ref{Thm:GR-EvenCycles}, there is a monochromatic even cycle of length at least $t - 1$ contained within $G \setminus T$. Since $T_{i} \neq \emptyset$ for all $i$, this cycle can easily be extended to a monochromatic copy of $P_{t}^{+}$, for a contradiction. 

Next we show that $|T_{i}| \leq \flr{\frac{t}{2}} - 1$ for all $i$. Thus, suppose there are at most two sets, $T_{i}$ and $T_{j}$, with $\flr{\frac{t}{2}} \leq |T_{i}| \leq t$ and possibly $\flr{\frac{t}{2}} \leq |T_{j}| \leq t$ (noting that one of these sets, say $T_{j}$, may not be large). Any edge of color $i$ (or possibly $j$) within $G \setminus T$ would produce a monochromatic copy of $P_{t}^{+}$ so $G\setminus T$ contains no edge of color $i$ (or possibly $j$). Then by Lemma~\ref{Lemma:SmallOutside}, we have $|G \setminus (T_{i} \cup T_{j})| \leq n_{k - 1} - 1$ so
\beqs
n & = & |T_{i}| + |G \setminus (T_{i} \cup T_{j})|\\
~ & \leq & 2t + n_{k - 1} - 1\\
~ & < & n_{k},
\eeqs
a contradiction.
\end{proof}

Let $G' = G \setminus T$. Since $G'$ is a Gallai coloring, it follows from Theorem~\ref{Thm:G-Part} that there is a Gallai partition of $V(G')$. Suppose that the two colors appearing in the Gallai partition are red and blue. Let $m$ be the number of parts in this partition and choose such a partition where $m$ is minimized. Let $H_{1}, H_{2}, \dots, H_{m}$ be the parts of this partition, say with $|H_{1}| \geq |H_{2}| \geq \dots \geq |H_{m}|$. When the context is clear, we also abuse notation and let $H_{i}$ denote the vertex of the reduced graph corresponding to the part $H_{i}$.

If $2\leq m\leq 3$, then by the minimality of $m$, we may assume $m=2$. Let $H_1$ and $H_2$ be the corresponding parts. Suppose all edges from $H_1$ to $H_2$ are red. If $|H_1|\geq \lceil t/2\rceil$ and $|H_2|\geq \lceil t/2\rceil$, then to avoid creating a red copy of $P_{t}^{+}$, there is no red edge in each $H_i$ with $i=1,2$ and the corresponding subset of $T$ is also empty, say $T_{1} = \emptyset$. This means that, by Lemma~\ref{Lemma:SmallOutside}, we have
\beqs
|G| & = & |T| + |H_{1}| + |H_{2}| \\
~ & \leq & |H_{1} \cup T_{2}| + |H_{2} \cup T_{3} \cup T_{4} \cup \dots \cup T_{k}|\\
~ & \leq & 2(n_{k - 1} - 1)\\
~ & < & n,
\eeqs
a contradiction. If $|H_1|\leq \lceil t/2\rceil-1$ and $|H_2|\leq \lceil t/2\rceil-1$, then 
\beqs
|G| & = & |T| + |H_{1}| + |H_{2}|\\
~ & \leq & \frac{k + 1}{2} t \\
~ & < & n_{k},
\eeqs
a contradiction. We may therefore assume that $|H_1|\geq \lceil t/2\rceil$ and $|H_2|\leq \lceil t/2\rceil-1$. Adding $|H_{2}|$ to $T$ contradicts the maximality of $T$ and completes the proof when $2 \leq m \leq 3$.

From now on, we assume $m\geq 4$. Let $r$ be the number of parts of the Gallai partition with order at least $\lceil t/2\rceil$ and call these parts ``large'' while other parts are called ``small''. Then $|H_r|\geq \lceil t/2\rceil$ and $|H_{r + 1}|\leq \lceil t/2\rceil-1$. To avoid a monochromatic copy of $P_{t}^{+}$, there can be no monochromatic triangle among these $r$ large parts, leading to the following immediate fact.
\begin{fact}\label{Fact:Pathr5}
$r\leq 5$.
\end{fact}

The remainder of the proof is broken into cases based on the value of $r$.

\setcounter{case}{0}
\begin{case}\label{Case:Pathr5}
$r=5$.
\end{case}

In this case, $t=5$ since otherwise any monochromatic triangle in the reduced graph restricted to $H_{1}, H_{2}, \dots, H_{6}$ would yield a monochromatic copy of $P_{t}^{+}$. To avoid the same construction, the reduced graph on the parts $H_1,H_2,H_3,H_4,H_5$ must be the unique $2$-coloring of $K_{5}$ with no monochromatic triangle with two complementary monochromatic cycles with in red and blue respectively. To avoid creating a monochromatic copy of $P_{t}^{+}$, 
for each $i$ with $1\leq i\leq 5$, the part $H_i$ contains neither red edges nor blue edges.
Then, by Lemma~\ref{Lemma:SmallOutside},
$$
|G|=|T| + \sum_{i=1}^5|H_i|\leq 5[n_{k - 2} - 1]<n_{k},
$$
a contradiction.

\begin{case}\label{Case:Pathr4}
$r=4$.
\end{case}

To avoid monochromatic triangle in $K_4$, without loss of generality, the four largest parts must form one of two structures:
\begin{itemize}
\item Type $1$: There is a red cycle $H_1H_3H_4H_2H_1$ and a blue $2$-matching on the edges $H_1H_4$, and $H_2H_3$ in the reduced graph, or
\item Type $2$: There is a red path $H_3H_2H_1H_4$ and a blue path $H_1H_3H_4H_2$ in the reduced graph.
\end{itemize}

If $r=m=4$, then using Lemma~\ref{Lemma:SmallOutside},
$$
|G|=|T| + \sum_{i=1}^4|H_i|\leq 4[n_{k - 2} - 1]<n_{k},
$$
a contradiction. So suppose $m>r=4$. This proof focuses on the reduced graph. For Type $1$, the subgraph of the reduced graph restricted to $\{H_{1}, H_{2}, H_{3}, H_{4}\}$ is not a subgraph of the unique $2$-colored copy of $K_{5}$ containing no rainbow triangle. This means that the subgraph of the reduced graph restricted to $\{H_{1}, H_{2}, \dots, H_{5}\}$ contains a monochromatic triangle, leading to a monochromatic copy of $P_{t}^{+}$ in $G$, a contradiction.

For Type $2$, outside of $\{H_1,H_2,H_3,H_4\}$, there are small parts $H_5,H_6,\ldots,H_m$. We may therefore assume that for all $H_i$ with $5\leq i\leq m$, we have that the edges $H_3H_i$ and $H_4H_i$ are blue. To avoid a blue triangle, the edges $H_1H_i$ and $H_2H_i$ are red. By minimality of $t$, we have $t = 5$ since all parts $H_{i}$ for $i \geq 5$ have the same color on edges to $H_{j}$ for $j \leq 4$. To avoid creating a monochromatic copy of $P_{t}^{+}$, none of these large parts contains any red or blue edges. 
Then using Lemma~\ref{Lemma:SmallOutside}, 
$$
|G|=|T| + \sum_{i=1}^4|H_i|\leq 4[gr_{k-2}(K_3;P_{t}^{+}) - 1]+\lceil t/2\rceil-1<n,
$$
a contradiction.

\begin{case}\label{Case:Pathr3}
$r=3$.
\end{case}

The triangle in the reduced graph cannot be monochromatic so without loss of generality, suppose $H_1H_2,H_1H_3$ are red, and $H_2H_3$ is blue. To avoid a red or blue triangle, any remaining parts are partitioned into the following sets.
\begin{itemize}
\item Let $A$ be the set of parts outside $H_1,H_2,H_3$ such that each has blue edges to $H_1,H_3$ and red edges to $H_2$,
\item Let $B$ be the set of parts outside $H_1,H_2,H_3$ such that each has red edges to $H_2,H_3$ and blue edges to $H_1$, and
\item Let $C$ be the set of parts outside $H_1,H_2,H_3$ such that each has blue edges to $H_1,H_2$ and red edges to $H_3$.
\end{itemize}
Note that $|G|=|T| + |A|+|B|+|C|+|H_1|+|H_2|+|H_3|$. If $A$ contains a blue edge, then using a long blue path between $H_{2}$ and $H_{3}$ along with a triangle formed using the blue edge in $A$, there would be a blue copy of $P_{t}^{+}$. This and similar easy arguments lead to the following fact.
\begin{fact}\label{Fact:Path2}~
\bi
\item $A$ contains no red or blue edges,
\item $B$ contains no red edges,
\item $C$ contains no red or blue edges,
\item $H_{1}$ contains no red edges,
\item $H_{2}$ contains no red or blue edges, and
\item $H_{3}$ contains no red or blue edges.
\ei
\end{fact}
Furthermore, we have the following claim.

\begin{claim}\label{Claim:Path}
If $B\neq \emptyset$, then $A=\emptyset$ and $C=\emptyset$.
\end{claim}

\begin{proof}
Assume, to the contrary, that $B \neq \emptyset$ and either $A\neq \emptyset$ or $C\neq \emptyset$, without loss of generality, say $A\neq \emptyset$. there is an red edge between $A$ and $B$, then there is a red triangle among $A,B,H_2$. This red triangle together with the red edges between $H_1$ and $H_2$ forms a red copy of $P_{t}^{+}$, a contradiction.
If there is a blue edge between $A$ and $B$, then there is a blue triangle among $A,B,H_1$. This blue triangle together a with blue path of the form $H_1CH_{2}$ and the edges between $H_2$ and $H_3$ forms a blue copy of $P_{t}^{+}$, a contradiction.
\end{proof}

From Claim~\ref{Claim:Path}, if $B\neq \emptyset$, then $A=\emptyset$ and $C=\emptyset$. We can then regard $H_1\cup B$ and $H_2\cup H_3$ as two parts of a Gallai partition of $G$ and the edges between these parts are all red, which contradicts the minimality of $t$.

We may therefore assume that $B=\emptyset$. Then we have the following claim.

\begin{claim}\label{Claim:PathOne}
There is only one part in $A$, and there is only one part in $C$.
\end{claim}

\begin{proof}
The edges between any pair of parts must be either red or blue but neither $A$ nor $C$ contain any red or blue edges by Fact~\ref{Fact:Path2}.
\end{proof}

Then by Lemma~\ref{Lemma:SmallOutside},
\beqs
|G| & = & |T| + |A|+|B|+|C|+|H_1|+|H_2|+|H_3| \\
~ & \leq & 3[n_{k - 2} - 1]+2(\lceil t/2\rceil-1)\\
~ & < & n_{k}, 
\eeqs
a contradiction.

\begin{case}\label{Case:Pathr2}
$r=2$.
\end{case}

Suppose all edges from $H_1$ to $H_2$ are red. To avoid creating a monochromatic copy of $P_{t}^{+}$, there is no part outside $H_1$ and $H_2$ with red edges to all of $H_1\cup H_2$. 
Therefore, any remaining parts are partitioned into the following sets.
\begin{itemize}
\item Let $A$ be the set of parts outside $H_1,H_2$ with blue edges to $H_2$ and red edges to $H_1$,
\item Let $B$ be the set of parts outside $H_1,H_2$ with blue edges to $H_1 \cup H_2$,
\item Let $C$ be the set of parts outside $H_1,H_2$ with blue edges to $H_1$ and red edges to $H_2$.
\end{itemize}
Note that $|G|=|A|+|B|+|C|+|H_1|+|H_2|$. Then we have the following fact.

Furthermore, we have the following claims.
\begin{claim}\label{Claim:PathNoBigSets}
$|A|\leq \lceil t/2\rceil-1$ and $|C|\leq \lceil t/2\rceil-1$. 
\end{claim}
\begin{proof}
Assume, to the contrary, that $|A|\geq \lceil t/2\rceil$. Then there are two parts in $A$, say $A',A''$. If the edges from $A'$ to $A''$ are red, then there is a red triangle among $A',A'',H_1$, together with the edges between $H_1$ and $H_2$, there is a red $P_{t}^{+}$, a contradiction. If the edges from $A'$ to $A''$ are blue, then there is a blue triangle among $A',A'',H_2$, there is a blue $P_{t}^{+}$ since $|A| \geq \lceil t/2\rceil$, a contradiction.
Similarly, there is only one part in $C$.
\end{proof}

\begin{claim}\label{Claim:PathSumSmall}
$|A|+|B|\leq t-1$ and $|B|+|C|\leq t-1$.
\end{claim}
\begin{proof}
Assume, to the contrary, that $|A|+|B|\geq t$. Then there are at least three parts in $A\cup B$. 
Since both $A$ and $B$ have blue edges to $H_{2}$, there can be no blue edges within $A \cup B$. This means there are only red edges appearing between the parts of the Gallai partition within $A \cup B$. With at least $t$ verices, there is a red copy of $P_{t}^{+}$ within $A \cup B$ for a contradiction.
\end{proof} 

From Claims~\ref{Claim:PathNoBigSets} and~\ref{Claim:PathSumSmall}, we have $|A|\leq \lceil t/2\rceil-1$, $|C|\leq \lceil t/2\rceil-1$, $|A|+|B|\leq t-1$, and $|B|+|C|\leq t-1$. By Lemma~\ref{Lemma:SmallOutside}, we have $|H_i|\leq gr_{k-1}(K_3;P_{t}^{+}) - 1$ with $i=1,2$, and $|H_1|+|C| \leq gr_{k-1}(K_3;P_{t}^{+}) - 1$ and $|H_1|+|A|+|B|\leq gr_{k-1}(K_3;P_{t}^{+}) - 1$. Again using Lemma~\ref{Lemma:SmallOutside}, we get
$$
|G|=|T| + |A|+|B|+|C|+|H_1|+|H_2| \leq 2(n_{k - 1} - 1)<n_{k},
$$
a contradiction. 

\begin{case}\label{Case:Pathr1}
$r=1$.
\end{case}

Let $A$ be the set of parts with blue edges to $H_1$, and $B$ be the set of parts with red edges to $H_1$.

If $|A|\geq \flr{\frac{t}{2} }$ and $|B|\geq \flr{\frac{t}{2}}$, then it follows from Claim~\ref{Claim:PathSumSmall} that $|A|\leq t+2$ and $|B|\leq t+2$. Since $A$ and $B$ are each large, it follows that there are neither red edges nor blue edges in $H_1$, and hence $|H_1|\leq n_{k - 2} - 1$. Then by Lemma~\ref{Lemma:SmallOutside},
$$
|G|= |T| + |H_1|+|A|+|B|\leq n_{k - 2}-1+(2t+4)<n_{k},
$$
a contradiction.

If $|A|\geq \flr{\frac{t}{2} }$ and $|B|\leq \flr{\frac{t}{2} } -1$, then it follows from Claim~\ref{Claim:Path} that $|A|\leq t+2$. Since $A$ is big, it follows that there are no red edges in $H_1$, and hence $|H_1|\leq n_{k - 1}-1$.
By Lemma~\ref{Lemma:SmallOutside}, 
$$
|G|=|T| + |H_1|+|A|+|B|\leq n_{k - 1} - 1 +(t+2)+(\flr{\frac{t}{2} } -1)<n_{k},
$$
a contradiction.

Finally if $|A|, |B| \leq \flr{ \frac{t}{2}} - 1$, both sets can be added to $T$, contradicting the maximality of $T$.

\begin{case}\label{Case:Pathr0}
$r=0$.
\end{case}

In this case, we have $|H_i|\leq \lceil t/2\rceil-1$ for all $i$ with $1 \leq i \leq m$. We need only consider the case $k=3$ since the parts are too small to contain a monochromatic copy of $P_{t}^{+}$ in a color other than red or blue. 
Then $n=5t-4$ so $|G'| \geq 4t - 2$. Note that this means there are at least $9$ parts in the Gallai partition of $G'$, so in particular, there must be either a red or blue triangle in the reduced graph.

Suppose first that there is both a red triangle and blue triangle in $G'$. Select one such triangle in each color and remove their vertices. By deleting at most six vertices, we still have at least $4t-8$ vertices in $G'$. Recall that by minimality of $m$, the edges of $G$ restricted to either red or blue induce a connected subgraph. 
From Theorem~\ref{Thm:GR-EvenCycles}, if $t-2$ is even, then $gr_{k}(K_3;C_{t-2})\leq 3t-9$ and if $t-3$ is even, then $gr_{k}(K_3;C_{t-3})\leq 3t-9$. This means $G'$ contains an even cycle $C_{t-2}$ or $C_{t-3}$. This cycle together with the deleted triangle form a red $P_{t}^{+}$ or blue $P_{t}^{+}$ (since each color induces a connected subgraph), a contradiction.

Thus, suppose that there is a red triangle but no blue triangles in $G'$. Choose an arbitrary set $H_{i}$ and let $G_{R}$ be the set of vertices with red edges to $H_{i}$ and $G_{B}$ be the set of vertices with blue edges to $H_{i}$. Note that $G_{B}$ contains no blue edge so if $|G_{B}| \geq t$, then since all parts have order at most $\ceil{t/2} - 1$, there are at least $3$ parts in $G_{B}$ and all red edges in between these parts, creating a red copy of $P_{t}^{+}$. Thus, $|G_{B}| \leq t - 1$.

Since $H_{i}$ was chosen arbitrarily, this is true about every such part. Let $D$ be the assumed red triangle. This means that every vertex in $G' \setminus D$, say $v \in H_{i}$, has red degree at least
$$
n_{k} - |T| - 3 - |H_{i}| - (t - 1) \geq \frac{|G' \setminus D|}{2}.
$$
By Dirac's Theorem \cite{MR0047308}, there exists a red Hamiltonian cycle within $G' \setminus D$. This cycle along with the red triangle $D$ (since the red subgraph is connected), produces a red copy of $P_{t}^{+}$, a contradiction to complete the proof.
\end{proof}


\end{document}